\documentclass[12pt]{amsart}
\usepackage{}

\usepackage{amsmath}
\usepackage{amsfonts}
\usepackage{amssymb}
\usepackage[all]{xy}           %xypic macro for latex2.09

\usepackage{bbding}
\usepackage{txfonts}
\usepackage{amscd}

\usepackage[shortlabels]{enumitem}
\usepackage{ifpdf}
\ifpdf
  \usepackage[colorlinks,final,backref=page,hyperindex]{hyperref}
\else
  \usepackage[colorlinks,final,backref=page,hyperindex,hypertex]{hyperref}
\fi
\usepackage{tikz}
\usepackage[active]{srcltx}

\makeatletter

\newtheorem{thm}{Theorem}[section]

\newtheorem{pro}[thm]{Proposition}
\newtheorem{ex}[thm]{Example}
\newtheorem{rmk}[thm]{Remark}
\newtheorem{defi}[thm]{Definition}

\newcommand {\emptycomment}[1]{}

\newcommand{\lon }{\,\rightarrow\,}
\newcommand{\be }{\begin{equation}}
\newcommand{\ee }{\end{equation}}

\newcommand{\g}{\mathfrak g}
\newcommand{\h}{\mathfrak h}

\newcommand{\huaB}{\mathcal{B}}%{{\mathcal{E}}}%{\mathcal{B}}

\newcommand{\frkk}{\mathfrak k}

\newcommand{\dM}{\mathrm{d}}

\newcommand{\Id}{{\rm{Id}}}

\newcommand{\br}[1]{   [ \cdot,    \cdot  ]   }

\newcommand{\Hom}{\mathrm{Hom}}

\newcommand{\Ad}{\mathrm{Ad}}

\newcommand{\gl}{\mathfrak {gl}}

\newcommand{\GL}{\mathrm{GL}}

\begin{document}

\title{The van Est theorem for Rota-Baxter Lie groups}

\author{Jun Jiang}
\address{Department of Mathematics, Jilin University, Changchun 130012, Jilin, China}
\email{junjiang@jlu.edu.cn}

%\date{\today}

\begin{abstract}
Rota–Baxter Lie groups are the group-level counterparts of Rota–Baxter Lie algebras of weight $1$. In this paper, we first establish a correspondence between representations of Rota–Baxter Lie groups and representations of Rota–Baxter Lie algebras of weight $1$ via differentiation and integration. We then develop a cohomology theory for Rota–Baxter Lie groups with coefficients in a representation by constructing an explicit cochain complex in arbitrary degrees. We justify this theory by proving a van Est type theorem, which relates the cohomology of a Rota–Baxter Lie group to that of its Rota–Baxter Lie algebra of weight $1$. As an application, we prove that equivalence classes of central extensions of a Rota–Baxter Lie group are classified by the second cohomology group.
\end{abstract}



\maketitle

%\tableofcontents

%\setcounter{section}{0}

\allowdisplaybreaks

%\end{document}

\section{Introduction}
Rota-Baxter operators on associative algebras were introduced by G. Baxter in his study of fluctuation
theory~\cite{Bax} and were developed in combination theory by Rota and
Cartier~\cite{Ra1, Ra2, Car}. They have since found applications in
Connes-Kreimer's renormalization theory in quantum field
theory~\cite{CK}. On a Lie algebra, a Rota-Baxter operator of weight $\lambda$ is a linear map satisfying a quadratic relation that is the Lie algebraic analogue of the
Rota-Baxter relation on associative algebras, the latter being an abstraction of the integration by parts formula. These operators are closely related to the classical Yang-Baxter equation and its variants: weight $0$ operators correspond to skew-symmetric solutions of the classical Yang-Baxter equation~\cite{STS}, while weight $1$ operators correspond to the modified Yang-Baxter equation and to factorizable Lie bialgebras~\cite{LS}. Lie groups with zero Schouten curvature also induce weight $1$ operators~\cite{KM}.

The group-level counterpart is the notion of a Rota-Baxter operator on a Lie group, introduced in~\cite{GLS, JSZ}. It is a smooth map of the group that is compatible with the group multiplication in a twisted sense. Such operators produce solutions of the
Yang-Baxter equation~\cite{BV, J}, and they are closely tied to group factorizations, skew braces and integrable systems~\cite{GLS, BNMY, CS, NM}. The Lie group and Lie algebra notions are linked by Lie's classical
differentiation-integration correspondence. Differentiating a
Rota-Baxter operator on a Lie group at the identity yields a
Rota-Baxter operator of weight $1$ on its Lie algebra. Conversely, a weight $1$ Rota-Baxter operator on a Lie algebra integrates to a local Rota-Baxter operator on the corresponding Lie group. However, as is typical for integration problems, the passage from local to global is in
general obstructed~\cite{JSZ, JSZ1}. Thus, Rota–Baxter Lie groups carry global information that is not captured by their infinitesimal counterparts, making it natural to investigate their cohomological properties.

The aim of this paper is to develop a cohomology theory for Rota–Baxter Lie groups and to relate it to the cohomology of Rota–Baxter Lie algebras of weight $1$ through a van Est type theorem. As an application, we establish a classification of central extensions of Rota–Baxter Lie groups in terms of the second cohomology group.

Cohomology theories for Rota–Baxter structures have been extensively studied in connection with deformation theory and extension problems. For Rota–Baxter operators of weight $0$, the cohomology and deformation theory of Rota-Baxter operators of weight $0$ on Lie algebras have been developed in \cite{TBGS}. Cohomology theories for relative Rota–Baxter Lie algebras, together with their applications to abelian extensions and higher algebraic structures, have also been investigated in \cite{JS}. See \cite{LaST,Sheng} for further developments concerning deformations, cohomology and homotopy theory. For arbitrary weights, Das \cite{DAS} introduced a cohomology theory for weighted Rota–Baxter Lie algebras with coefficients in a representation and investigated its applications to abelian extensions and formal deformations. An important feature of this theory is that its cochain complex is constructed as the mapping cone of a cochain map between the Chevalley–Eilenberg complexes of the underlying Lie algebra and its descendant Lie algebra. This construction naturally incorporates both the Lie algebra structure and the compatibility conditions imposed by the Rota–Baxter operator.

At the group level, a low-dimensional cohomology theory for Rota–Baxter groups was developed in \cite{DR} to study extensions and automorphisms, with the second cohomology group classifying abelian extensions. This approach was extended to relative Rota–Baxter groups in \cite{BRS}, where a corresponding extension-classification theorem was established. However, these cohomology theories are formulated only in low degrees. This motivates the construction of a cohomology theory for Rota–Baxter Lie groups in arbitrary degrees, together with a systematic comparison with the cohomology of their infinitesimal counterparts.

Motivated by the cohomology theory of weighted Rota–Baxter Lie algebras, we construct a cochain complex for Rota–Baxter Lie groups with coefficients in a representation by means of a mapping cone. The construction is based on a cochain map relating the differentiable group cohomology of the underlying Lie group to that of its descendant Lie group. In this way, we obtain a cohomology theory in arbitrary degrees, which simultaneously incorporates the underlying group structure and the Rota–Baxter operator. Moreover, the mapping-cone description gives rise to a long exact sequence relating the cohomology of Rota–Baxter Lie groups, the cohomology associated with Rota–Baxter operators, and ordinary Lie group cohomology.

A natural question is how the resulting cohomology theory is related to that of the corresponding Rota–Baxter Lie algebra. The classical van Est theorem provides a fundamental connection between the differentiable cohomology of a Lie group and the Chevalley–Eilenberg cohomology of its Lie algebra~\cite{Van}. More precisely, under suitable connectivity assumptions, it establishes an isomorphism between these cohomology groups in a range of degrees. This comparison has important consequences in Lie theory. In particular, by identifying second cohomology classes on the Lie algebra side with suitable central extensions and integrating these classes to the group level, the van Est theorem enters a cohomological proof of Lie's third theorem, namely that every finite-dimensional Lie algebra can be integrated to a Lie group. In the context of Rota–Baxter structures, van Est type theorems have been established for the cohomology of relative Rota–Baxter operators on Lie groups and Lie algebras \cite{JSZ}, as well as for difference Lie groups and their corresponding Lie algebras \cite{JLS}. Motivated by these developments, we first establish a correspondence between representations of Rota–Baxter Lie groups and representations of Rota–Baxter Lie algebras of weight $1$ via differentiation and integration. We then prove a van Est type theorem relating the cohomology of a Rota–Baxter Lie group to that of its associated Rota–Baxter Lie algebra of weight $1$.

Finally we study central extensions of Rota-Baxter Lie groups as applications. We show that central extensions of Rota-Baxter Lie groups are classified by the second cohomology group given above. However, it remains unclear whether the cochain complex introduced here can also be used to classify abelian extensions of Rota–Baxter Lie groups.

The paper is organized as follows. In Section \ref{repdi}, we establish a correspondence between representations of Rota-Baxter Lie groups and representations of Rota-Baxter Lie algebras of weight $1$ via differentiation and integration (Proposition \ref{diffr}, Proposition \ref{intrep}). In Section \ref{seccoh}, we use mapping cone construction to develop a cohomology theory for Rota–Baxter Lie groups with coefficients in a representation, by constructing an explicit cochain complex in arbitrary degrees. The relation among cohomology groups of Rota-Baxter Lie groups, cohomology groups of Rota-Baxter operators and cohomology groups of Lie groups is described by a long exact sequence (Theorem \ref{logco}). In Section \ref{secvan}, we prove a van Est type theorem
relating the cohomology of Rota-Baxter Lie groups to that of Rota-Baxter Lie algebras of weight $1$ (Theorem \ref{van1}). In Section \ref{secch2}, we classify central extensions of Rota-Baxter Lie groups in terms of the second cohomology group introduced in Section \ref{seccoh} (Theorem \ref{cenH2}).
%\vspace{2mm}

%In this paper, we work over an field $\R$ and all the vector spaces are over $\R$ and finite-dimensional.

\vspace{2mm}
\noindent
{\bf Acknowledgements. } This research is supported by NSFC (12401076).

\section{Differentiation and Integration between Representations of Rota-Baxter Lie Groups and Rota-Baxter Lie Algebras}\label{repdi}
In this section, we first recall the notions of representations of Rota-Baxter Lie groups and Rota-Baxter Lie algebras. Then we establish the relationship between representations of Rota-Baxter 
Lie groups and representations of Rota-Baxter Lie algebras via diﬀerentiation and integration.
\begin{defi}{\rm(\cite{GLS})}
	 \begin{enumerate}
		\item
		Let $(G, \cdot_G, e_G)$ be a Lie group.  A smooth map $\huaB: G\to G$ is called  a {\bf Rota-Baxter operator} on $G$ if the following equality holds:
		\begin{equation}\label{defiB}
		\huaB(g)\cdot_G\huaB(h)=\huaB(g\cdot_G\Ad_{\huaB(g)}h), \quad \forall g, h\in G.
		\end{equation}
		The pair $(G, \huaB)$ is called a Rota-Baxter Lie group.
	\item Let $(G, \huaB)$  and $(G', \huaB')$ be Rota-Baxter Lie groups. A smooth map $\Psi:G\to G'$ is called a {\bf Rota-Baxter Lie group homomorphism} if $\Psi$ is a Lie group homomorphism and 
    \begin{equation*}
	\Psi\circ \huaB=\huaB'\circ \Psi.
	\end{equation*}
	 \end{enumerate}
\end{defi}
\begin{ex}
	Let $(G, \cdot_G, e_G)$ be a Lie group. Then $(G, (\bullet)^{-1})$ is a Rota-Baxter Lie group, where $(\bullet)^{-1}:G\to G$ is the inverse map of $G$.
\end{ex}

\begin{ex}
	Let $V$ be a finite dimensional vector space and $T: V\to V$ a linear map. Then $(V, T)$ is a Rota-Baxter Lie group.
\end{ex}

In the following, we recall the infinitesimal structure of Rota-Baxter Lie groups.
\begin{defi}{\rm(\cite{BGN}, \cite{BD})}
	Let $(\g, [\cdot, \cdot]_\g)$ be a Lie algebra. A linear map $B:\g\to\g$ is called a {\bf Rota-Baxter operator of weight $1$} on $\g$ if 
	\begin{equation*}
		[B(x), B(y)]_\g=B([B(x), y]_\g+[x, B(y)]_\g+[x, y]_\g), \quad \forall x, y\in\g.
	\end{equation*}
	The pair $(\g, B)$ is called a {\bf Rota-Baxter Lie algebra of weight $1$.}
\end{defi}

\begin{rmk}
	In the sequel, we will call Rota-Baxter Lie algebras of weight $1$ simply Rota-Baxter Lie algebras when there is no ambiguity.
\end{rmk}

There is a differentiation relation between Rota-Baxter Lie groups and Rota-Baxter Lie algebras as follows.
\begin{thm}{\rm(\cite{GLS})}
	Let $(G, \huaB)$ be a Rota-Baxter Lie group. Denote by $(\g, [\cdot, \cdot]_\g)$ the Lie algebra of $(G, \cdot_G, e_G)$. Then $(\g, \huaB_{*e_G})$ is a Rota-Baxter Lie algebra.
\end{thm}
Conversely, every Rota-Baxter operator on a Lie algebra can be locally integrated to a Rota-Baxter operator on the corresponding connected and simply connected Lie group, but there sometimes exist obstructions to extending a local Rota-Baxter operator to a global one. For the local integration of Rota-Baxter operators, as well as criteria for determining whether such a local integration can be extended globally, we refer the reader to \cite{JSZ} and \cite{JSZ1}.

Recall representations of Rota-Baxter Lie groups and Rota-Baxter Lie algebras as follows. For further details, see \cite{DR}, \cite{DAS} and \cite{JS}.

	\begin{defi}{\rm(\cite{DR})}\label{defirep}
	Let $V$ be a vector space and $T: V\to V$ a linear map. A {\bf representation} of a Rota-Baxter Lie group $(G, \huaB)$ on $(V, T)$ is a Lie group homomorphism $\Theta$ from $G$ to $\GL(V)$ such that 
	\begin{equation}\label{rep}
		\Theta(\huaB(a))T(u)=T\Big(\Theta(a\cdot_G\huaB(a))u+\Theta(a\cdot_G\huaB(a))T(u)-\Theta(\huaB(a))T(u)\Big),\quad \forall a\in G, ~u\in V.
	\end{equation}
	We denote a representation by $(V, T, \Theta)$.
\end{defi}

\begin{rmk}
In \cite{DR}, the notion of a representation of a Rota-Baxter Lie group $(G, \huaB)$ introduced by the authors requires $V$ to be a right $G$-module. By converting between left and right $G$-modules, one can show that the definition given above is equivalent to the one in \cite{DR}.
\end{rmk}

In the following, we construct the semidirect product in the context of Rota-Baxter Lie groups in order to characterize representations of Rota-Baxter Lie groups.

\begin{pro}\label{semirb}
Let $(G, \huaB)$ be a Rota-Baxter Lie group and $\Theta: G\to \GL(V)$ a representation of $G$. Let $T: V\to V$ be a linear map. Consider the semidirect product Lie group $G\ltimes_\Theta V$, whose multiplication $\cdot_\ltimes$ is given by
	\begin{equation*}
		(a, u)\cdot_{\ltimes}(b, v)=(a\cdot_G b, u+\Theta(a)v), \quad \forall a, b\in G, u, v\in V,
	\end{equation*}
Define $\mathrm{B}: G\times V\to G\times V$ by
	\begin{equation*}
		\mathrm{B}(a, u)=\Big(\huaB(a), \Theta(\huaB(a))T(\Theta(\huaB(a)^{-1}\cdot_G a^{-1})u)\Big),\quad \forall a\in G, u\in V.
	\end{equation*}
Then $(G\ltimes_\Theta V, \mathrm{B})$ is a Rota-Baxter Lie group if and only if $(V, \Theta, T)$ is a representation of $(G, \huaB)$.
\end{pro}
\begin{proof}
	On the one hand, for any $a, b\in G, u, v\in V$, we have
	\begin{equation*}
		\mathrm{B}(a, u)\cdot_{\ltimes}\mathrm{B}(b, v)=\Big(\huaB(a)\cdot_G\huaB(b), \Theta(\huaB(a))T(\Theta(\huaB(a)^{-1}\cdot_G a^{-1})u)+\Theta(\huaB(a)\cdot_G\huaB(b))T(\Theta(\huaB(b)^{-1}\cdot_G b^{-1})v)\Big).
	\end{equation*}
On the other hand, 
\begin{align*}
&\quad \mathrm{B}\Big((a, u)\cdot_{\ltimes}\mathrm{B}(a, u)\cdot_{\ltimes}(b, v)\cdot_{\ltimes}\mathrm{B}(a, u)^{-1}\Big)\\
&=\Big(\huaB(a\cdot_G\huaB(a)\cdot_G b\cdot_G\huaB(a)^{-1}), \Theta(\huaB(a\cdot_G\huaB(a)\cdot_G b\cdot_G\huaB(a)^{-1}))T\Big(\Theta(\huaB(b)^{-1}\cdot_G b^{-1}\cdot_G\huaB(a)^{-1}\cdot_G a^{-1})u\\
&\quad+\Theta(\huaB(b)^{-1}\cdot_G b^{-1})T(\Theta(\huaB(a)^{-1}\cdot_G a^{-1})u)+\Theta(\huaB(b)^{-1}\cdot_G b^{-1})v-\Theta(\huaB(b)^{-1})T(\Theta(\huaB(a)^{-1}\cdot_G a^{-1})u)\Big) \Big).
\end{align*}
By \cite[Proposition 3.5]{JLS}, $\huaB(b)^{-1}=\huaB(\huaB(b)^{-1}\cdot_G b^{-1}\cdot_G\huaB(b))$, let $w=\Theta(\huaB(a)^{-1}\cdot_G a^{-1})u$. Then to prove $\mathrm{B}$ is a Rota-Baxter operator if and only if 
\begin{align}\label{eqrepth}
&\quad\Theta(\huaB\Big(\huaB(b)^{-1}\cdot_G b^{-1}\cdot_G\huaB(b))\Big)T(w)\\
\nonumber&=T\Big(\Theta(\huaB(b)^{-1}\cdot_G b^{-1})w+\Theta(\huaB(b)^{-1}\cdot_G b^{-1})T(w)-\Theta(\huaB(b)^{-1})T(w)\Big).
\end{align}
By \cite[Proposition 3.5]{JLS}, for any $g\in G$, there exists a unique $h\in G$ such that 
\begin{equation*}
	g=\huaB(h)^{-1}\cdot_G h^{-1}\cdot_G\huaB(h).
\end{equation*}
Thus \eqref{eqrepth} holds if and only if for any $g\in G, r\in V$ such that
\begin{equation*}
\Theta(\huaB(g))T(r)=T\Big(\Theta(g\cdot_G\huaB(g))r+\Theta(g\cdot_G\huaB(g))T(r)-\Theta(\huaB(g))T(r)\Big).
\end{equation*}
We finish the proof.
\end{proof}

Let $(G, \huaB)$ be a Rota-Baxter Lie group, with Lie algebra $\g=T_{e_G}G$ and induced infinitesimal Rota-Baxter operator $B:=\huaB_{*e_G}: \g \to \g$. Let $\widetilde{\Ad} : G \to \GL(\g)$ denote the adjoint representation of the Lie group $(G, \cdot_G, e_G)$ on its Lie algebra $(\g, [\cdot,\cdot]_\g)$. As we now show, this classical construction provides a fundamental example of a representation in the sense of Definition \ref{defirep}.

\begin{pro}\label{repad}
	For a Rota-Baxter Lie group $(G, \huaB)$, $(\g, B, \widetilde{\Ad})$ is a representation of $(G, B)$ on $(\g, B)$.
\end{pro} 
\begin{proof}
According to \cite[Proposition 3.15]{JSZ}, $\huaB(a)^{-1}=\huaB(\Ad_{\huaB(a)^{-1}}a^{-1})$ for any $a\in G$. Denote the exponential map 
of the Lie group $G$ by $\exp$. By the fact that $\huaB$ is a Rota-Baxter operator, for any $a\in G, u\in V$, we have
\begin{align*}
\widetilde{\Ad}_{\huaB(a)}(B(u))	&=\frac{d}{dt}|_{t=0}\Ad_{\huaB(a)}\huaB(\exp(tu))\\
	&=\frac{d}{dt}|_{t=0}\huaB\Big(\Ad_{a\cdot_G\huaB(a)}(\exp(tu)\cdot_G\huaB(\exp(tu)))\cdot_G\Ad_{\huaB(a)}\huaB(\exp(tu))^{-1}\Big)\\
	&=B\Big(\frac{d}{dt}|_{t=0}\Ad_{a\cdot_G\huaB(a)}(\exp(tu)\cdot_G\huaB(\exp(tu)))+\frac{d}{dt}|_{t=0}\Ad_{\huaB(a)}\huaB(\exp(tu))^{-1}\Big)\\
	&=B(\widetilde{\Ad}_{a\cdot_G\huaB(a)}(u+B(u))-\widetilde{\Ad}_{\huaB(a)}B(u)),
	\end{align*}
	which means that $(\g, B, \widetilde{\Ad})$ is a representation of $(G, \huaB)$ on $(\g, B)$.
\end{proof}
\begin{defi}{\rm(\cite{DAS})}
	Let $V$ be a finite dimensional vector space and $T:V\to V$ a linear map. A {\bf representation} of $(\g, B)$ on $(V, T)$ is a Lie algebra representation $\theta: \g\to \gl(V)$ satisfying
	\begin{equation}\label{defth}
		\theta(B(x))T(u)=T(\theta(B(x))u+\theta(x)T(u)+\theta(x)u)\quad \forall x\in\g, ~u\in V.
	\end{equation}
	We denote a representation by $(V, T, \theta)$.
\end{defi}

Next, we consider the relationship between representations of Rota-Baxter Lie groups and representations of Rota-Baxter Lie algebras.

\begin{pro}\label{diffr}
	Let $(V, T, \Theta)$ be a representation of a Rota-Baxter Lie group $(G, \huaB)$. Then $(V, T, \theta)$ is a representation of the Rota-Baxter Lie algebra $(\g, \huaB_{*e_G})$, where $\g=T_{e_G}G$ and $\theta=\Theta_{*e_G}$.
\end{pro}
\begin{proof}
	Denote the exponential map of the Lie group $G$ by $\exp$. Then we have
	\begin{align*}
		\theta(B(x))T(u)&=\frac{d}{dt}|_{t=0}\Theta(\huaB(\exp(tx)))T(u)\\
		&=T\Big(\Theta(\exp(tx)\cdot_G\huaB(\exp(tx)))u+\Theta(\exp(tx)\cdot_G\huaB(\exp(tx)))T(u)-\Theta(\huaB(\exp(tx)))T(u)\Big)\\
		&=T(\theta(B(x))u+\theta(x)T(u)+\theta(x)u),
	\end{align*}
	for any $x\in\g, u\in V$. Thus $(V, T, \theta)$ is a representation of the Rota-Baxter Lie algebra $(\g, \huaB_{*e_G})$.
\end{proof}

Let $(V, T, \theta)$ be a representation of a Rota-Baxter Lie algebra $(\g, B)$, and let $(G, \cdot_G, e_G)$ denote the connected and simply connected Lie group corresponding to $(\g, [\cdot, \cdot]_\g)$. Suppose that there exists a Rota-Baxter operator $\huaB$ on $G$ satisfying $\huaB_{*e_G}=B$. We shall show that, under this assumption, the representation $(V, T, \theta)$ can be integrates to a representation of $(G, \huaB)$.
\begin{pro}\label{intrep}
	With the above notations, let $\Theta:G\to \GL(V)$ be the representation of $G$ corresponding to $\theta:\g\to\gl(V)$. Then $(V, T, \Theta)$ is a representation of $(G, \huaB)$.
\end{pro}
\begin{proof}
Since $\huaB$ is a Rota-Baxter operator on $G$, the set 	
\begin{equation*}
H=\{(a\cdot_G\huaB(a), \huaB(a))|\forall a\in G\}\subset G\times G
\end{equation*}
 is a Lie subgroup of the direct product Lie group $G\times G$, with Lie algebra 
 \begin{equation*}
 \h=\{(\Id+B)(x), B(x)|\forall x\in \g\}.
 \end{equation*}
  Moreover, the map $\kappa:H\to G$ defined by $\kappa(a\cdot_G\huaB(a), \huaB(a))=a$ is a homeomorphism, consequently $H$ is a connected and simply connected Lie group. 

Define the vector space $W=\{((\Id+T)u, T(u))|\forall u\in V\}$. Consider the representation of the direct product Lie algebra $\g\oplus\g$ on $V\times V$, given by 
\begin{equation*}
\lambda:\g\oplus\g\to\gl(V\times V), \quad	\lambda(x, y)(u, v)=(\theta(x)u, \theta(y)v), \quad x, y\in\g, ~\forall u, v\in V,
\end{equation*}
and let 
\begin{equation*}
\Lambda: G\times G\to \GL(V\times V), \quad \Lambda(a, b)(u, v)=(\Theta(a)u, \Theta(b)v), \quad \forall a, b\in G,~u, v\in V.
\end{equation*}
be the corresponding representation of $G\times G$ on $V\times V$. By \eqref{defth}, the restriction $\lambda|_{\h}$ preserves $W$, and hence induces a representation of $\h$ on $W$. Since $H$ is connected and simply connected, this representation of $\h$ integrates to a representation of $H$ on $W$, which is precisely $\Lambda|_H$. It follows that
\begin{equation*}
\Lambda|_{H}(a\cdot_G\huaB(a), \huaB(a))((\Id+T)u, T(u))=(\Theta(a\cdot_G\huaB(a))(\Id+T)u, \Theta(\huaB(a))T(u))\in W,
\end{equation*}
which gives
\begin{equation*}
\Theta(\huaB(a))T(u)=T\Big(\Theta(a\cdot_G\huaB(a))u+\Theta(a\cdot_G\huaB(a))T(u)-\Theta(\huaB(a))T(u)\Big).
\end{equation*} 
Thus $(V, T, \Theta)$ is a representation of $(G, \huaB)$.
\end{proof}

\section{Cohomology of Rota-Baxter Lie groups}\label{seccoh}
In this section, we introduce a cohomology theory for a Rota-Baxter Lie group. The relation between cohomology groups of Rota-Baxter operators and cohomology groups of Rota-Baxter Lie groups is given by a long exact sequence.

First we  recall the normalized cohomology of a Lie group $G$ with coefficients in a representation
$\Theta:G\to \GL(V)$ (see e.g. \cite{EM}\cite{Wei}). A smooth map $\alpha_n: \underbrace{G\times\cdots\times G}_n\lon V$ is called an $n$-normalized cochain in $\alpha_n(g_1, \cdots, g_n)=0$ when any one of elements $g_i=e_G$. Denote the space of $n$-normalized cochains by $C^{k}(G, V)$. The space $C^{n}(G, V)$ is an abelian group. The coboundary operator $\dM^\Theta: C^{n}(G, V)\rightarrow C^{n+1}(G, V)$ is defined by
\begin{eqnarray*}
	\dM^\Theta(\alpha_n)(g_1,\cdots,g_n,g_{n+1})&=&\Theta(g_1)\alpha_n(g_2, \cdots, g_n, g_{n+1})\\
	&&+\sum_{i=1}^{n}(-1)^{i}\alpha_n(g_1,\cdots,g_{i-1}, g_ig_{i+1}, g_{i+2}, \cdots g_{n+1})\\
	&&+(-1)^{n+1}\alpha_{n}(g_1,\cdots,g_n).
\end{eqnarray*}
The corresponding $n$-th cohomology group is denoted by $H^{n}(G, V)$.
\begin{rmk}
It was proved in \cite{EM} that the normalized cohomology of a Lie group G with coefficients in a representation $\Theta: G\to\GL(V)$ is isomorphic to the usual cohomology.
\end{rmk}

Let $(G, \huaB)$ be a Rota-Baxter Lie group. Then $G$ has a new Lie group structure $\star: G\times G\to G$, defined by
\begin{equation*}
	a\star b=a\cdot_G\huaB(a)\cdot_G b\cdot_G\huaB(a)^{-1}, \quad \forall a, b\in G.
\end{equation*}
Let $(\g, [\cdot, \cdot]_\g)$ be the Lie algebra of $(G, \cdot_G, e_G)$ and $B=\huaB_{*e_G}$. Then the Lie algebra of $(G, \star, e_G)$ is $(\g, [\cdot, \cdot]_{B})$, where
\begin{equation*}
	[x, y]_B=[B(x), y]_\g+[x, B(y)]_\g+[x, y]_\g, \quad\forall x, y\in\g.
\end{equation*}
Denote by $G_\huaB$ the Lie group $(G, \star, e_G)$ and $\g_B$ the Lie algebra $(\g, [\cdot, \cdot]_B)$ respectively, they play important roles in the following. 
\begin{pro}\label{defirB}
Let $(V, T, \Theta)$ be a representation of $(G, \huaB)$. Define a smooth map $\Theta_{\huaB}: G_\huaB\to \GL(V)$ by 
\begin{equation}\label{repB}
\Theta_{\huaB}(a)u=(\Id_V+T)\Theta(\huaB(a))u-T\Big(\Theta(a\cdot_G\huaB(a))u\Big),\quad \forall a\in G, u\in V.
\end{equation}
Then $\Theta_{\huaB}$ is a representation of $G_\huaB$ on $V$.
\end{pro}
\begin{proof}
For any $a, b\in G, u\in V$, since 
	\begin{align*}
	\Theta_{\huaB}(a)\Theta_{\huaB}(b)u&=(\Id_V+T)\Theta(\huaB(a))(\Id_V+T)\Theta(\huaB(b))u-(\Id_V+T)\Theta(\huaB(a))T(\Theta(b\cdot_G\huaB(b))u)\\
	&\quad-T\Big(\Theta(a\cdot_G\huaB(a))(\Id_V+T)\Theta(\huaB(b))u\Big)+T\Big(\Theta(a\cdot_G\huaB(a))T(\Theta(b\cdot_G\huaB(b)))u\Big)\\
	&=(\Id_V+T)\Theta(\huaB(a)\huaB(b))u-T(\Theta(a\cdot_G\huaB(a))\Theta(b\cdot_G\huaB(b))u)\\
	&=\Theta_\huaB(a\star b)u,
	\end{align*} 
	then $\Theta_{\huaB}$ is a representation of $G_\huaB$ on $V$.
\end{proof}
\begin{rmk}
	In particular, for the representation $(\g, \huaB_{*e_G}, \widetilde{\Ad})$ of $(G, \huaB)$ from Proposition \ref{repad}, the map $\Theta_\huaB$ recovers the map introduced in \cite{JSZ} for studying the cohomology of Rota-Baxter operators on Lie groups with coefficient in the representation $(\g, \huaB_{*e_G}, \widetilde{\Ad})$.
\end{rmk}

Now we define a cohomology theory for Rota-Baxter operators on Lie groups. Let $(V, T, \Theta)$ be a representation of a Rota-Baxter Lie group $(G, \huaB)$. Deﬁne the space of 
1-normalized cochains $C^{1}(\huaB, T)$ to be $0$. For $n \geq 2$, deﬁne the space of $n$-normalized 
cochains $C^{n}(\huaB, T)$ to be $C^{n-1}(G, V)$.

\begin{defi}
	The cohomology of the cochain complex $(\oplus_{n=1}^{\infty}C^{n}(\huaB, T), \dM^{\huaB})$ is called the 
	cohomology of the Rota-Baxter operator $\huaB$ with coefficients in the representation $(V, T, \Theta)$, 
	where $\dM^\huaB$ is the coboundary operator for the Lie group $G_\huaB$ with coefficients in the representation
	$(V, \Theta_\huaB)$. The corresponding $n$-th cohomology group is denoted by $H^{n}(\huaB)$.
\end{defi}

Let $(V, T, \Theta)$ be a representation of a Rota-Baxter Lie group $(G, \huaB)$. For $n\geq 1$, define a map $c_n: C^{n}(G, V)\to C^{n+1}(\huaB, T)$ by
\begin{equation}\label{dec}
	c_n(f)(a_1, \cdots, a_n)=(-1)^{n}\Big((\Id_V+T)f(\huaB(a_1), \cdots, \huaB(a_n))-T(f(a_1\cdot_G\huaB(a_1), \cdots, a_n\cdot_G\huaB(a_n)))\Big),
\end{equation}
where $a_i\in G$. Assembling these maps, we obtain a map $c:\oplus_{i=1}^{+\infty}C^{i}(G, V)\to \oplus_{i=1}^{+\infty}C^{i+1}(\huaB, T)$, given by $c|_{C^{n}(G, V)}=c_n$.
\begin{pro}\label{chto}
The map $c:\oplus_{i=1}^{\infty}C^{i}(G, V)\to \oplus_{i=1}^{+\infty}C^{i+1}(\huaB, T)$ is a cochain map, i.e. we have $c\circ \dM^{\Theta}+\dM^{\huaB}\circ c=0$.
\end{pro}
\begin{proof}
For any $a_i\in G$ and $f\in C^{n}(G, V)$, by \eqref{repB} and \eqref{dec}, we have
\begin{align*}
&\quad\Big(c_{k+1}(\dM^{\Theta}f)+\dM^{\huaB}c_{k}(f)\Big)(a_1,\cdots, a_{k+1})\\
&=(-1)^{k+1}\Big((\Id_V+T)\dM^\Theta f(\huaB(a_1), \cdots, \huaB(a_{k+1}))-T(\dM^\Theta f(a_1\cdot_G\huaB(a_1), \cdots, a_{k+1}\cdot_G\huaB(a_{k+1})))\Big)\\
&\quad +\Theta_\huaB(a_1)c_k(f)(a_2, \cdots, a_{k+1})+\sum_{i=1}^{k}(-1)^{i}c_k(f)(a_1, \cdots, a_i\star a_{i+1}, \cdots, a_{k+1})+(-1)^{k+1}c_k(f)(a_1, \cdots, a_k)\\
&=(-1)^{k+1}(\Id_V+T)\Big(\Theta(\huaB(a_1))f(\huaB(a_2), \cdots, \huaB(a_{k+1}))+\sum_{i=1}^{k}(-1)^{i}f(\huaB(a_1), \cdots, \huaB(a_i)\cdot_G\huaB(a_{i+1}),\cdots,\huaB(a_{k+1}))\\
&\quad+(-1)^{k+1}f(\huaB(a_1),\cdots, \huaB(a_k))\Big)-(-1)^{k+1}T\Big(\Theta(a_1\cdot_G\huaB(a_1))f(a_2\cdot_G\huaB(a_2), \cdots, a_{k+1}\cdot_G\huaB(a_{k+1}))\\
&\quad+\sum_{i=1}^{k}(-1)^{i}f(a_1\cdot_G\huaB(a_1),\cdots,a_i\cdot_G\huaB(a_i)\cdot_G a_{i+1}\cdot_G\huaB(a_{i+1}),\cdots, a_{k+1}\cdot_G\huaB(a_{k+1}))\\
&\quad+(-1)^{k+1}f(a_1\cdot_G\huaB(a_1), \cdots, a_{k}\cdot_G\huaB(a_{k})))\Big)+(-1)^{k}(\Id_V+T)\Theta(\huaB(a_1))\Big((\Id_V+T)f(\huaB(a_2), \cdots, \huaB(a_{k+1}))\\
&\quad-T(f(a_2\cdot_G\huaB(a_2), \cdots, a_{k+1}\cdot_G\huaB(a_{k+1})))\Big)-(-1)^{k}T\Big(\Theta(a_1\cdot_G\huaB(a_1))\Big((\Id_V+T)f(\huaB(a_2), \cdots, \huaB(a_{k+1}))\\
&\quad-T(f(a_2\cdot_G\huaB(a_2), \cdots, a_{k+1}\cdot_G\huaB(a_{k+1})))\Big)\Big)+\sum_{i=1}^{k}(-1)^{i+k}\Big((\Id_V+T)f(\huaB(a_1), \cdots, \huaB(a_i)\cdot_G\huaB(a_{i+1}), \cdots, \huaB(a_k))\\
&\quad-T(f(a_1\cdot_G\huaB(a_1), \cdots, a_i\cdot_G\huaB(a_i)\cdot_G a_{i+1}\cdot_G\huaB(a_{i+1}),\cdots, a_k\cdot_G\huaB(a_k)))\Big)\\
&\quad+(-1)^{2k+1}\Big((\Id_V+T)f(\huaB(a_1), \cdots, \huaB(a_k))-T(f(a_1\cdot_G\huaB(a_1), \cdots, a_k\cdot_G\huaB(a_k)))\Big)\\
&=(-1)^{k}\Big((\Id_V+T)\Theta(\huaB(a_1))T\Big(f(\huaB(a_2), \cdots, \huaB(a_{k+1}))\Big)+T\Big(\Theta(a_1\cdot_G\huaB(a_1))f(a_2\cdot_G\huaB(a_2), \cdots, a_{k+1}\cdot_G\huaB(a_{k+1}))\Big)\\
&\quad-(\Id_V+T)\Theta(\huaB(a_1))T\Big(f(a_2\cdot_G\huaB(a_2), \cdots, a_{k+1}\cdot_G\huaB(a_{k+1}))\Big)\\
&\quad-T\Big(\Theta(a_1\cdot_G\huaB(a_1))(\Id_V+T)f(\huaB(a_2), \cdots, \huaB(a_{k+1}))\Big)\\
&\quad+T\Big(\Theta(a_1\cdot_G\huaB(a_1))T(f(a_2\cdot_G\huaB(a_2), \cdots, a_{k+1}\cdot_G\huaB(a_{k+1})))\Big)\Big).
\end{align*}
By the fact that 
\begin{equation*}
	\Theta(\huaB(a))T(u)=T\Big(\Theta(a\cdot_G\huaB(a))u+\Theta(a\cdot_G\huaB(a))T(u)-\Theta(\huaB(a))T(u)\Big), \quad \forall a\in G, ~u\in V,
\end{equation*}
we have 
\begin{equation*}
	(\Id_V+T)\Theta(\huaB(a_1))T(f(\huaB(a_2), \cdots, \huaB(a_{k+1})))=T(\Theta(a_1\cdot_G\huaB(a_1))(\Id_V+T)f(\huaB(a_2), \cdots, \huaB(a_{k+1}))),
\end{equation*}
and
\begin{align*}
&\quad(\Id_V+T)\Theta(\huaB(a_1))T\Big(f(a_2\cdot_G\huaB(a_2), \cdots, a_{k+1}\cdot_G\huaB(a_{k+1}))\Big)\\
&= T\Big(\Theta(a_1\cdot_G\huaB(a_1))f(a_2\cdot_G\huaB(a_2), \cdots, a_{k+1}\cdot_G\huaB(a_{k+1}))\Big)\\
&\quad+T\Big(\Theta(a_1\cdot_G\huaB(a_1))T(f(a_2\cdot_G\huaB(a_2), \cdots, a_{k+1}\cdot_G\huaB(a_{k+1})))\Big).
\end{align*}
Thus $c_{k+1}(\dM^\Theta f)+\dM^{\theta_\huaB}c_k(f)=0$, which means that $c$ is a cochain map.
\end{proof}

Let $(V, T, \Theta)$ be a representation of a Rota-Baxter Lie group $(G, \huaB)$. We define the space of $1$-cochains $C^{1}(G, \huaB, V, T)$ to be $C^1(G, V)$, and for $n\geq 2$, the space of $n$-cochains $C^{n}(G, \huaB, V, T)$ by
\begin{equation*}
	C^{n}(G, \huaB, V, T)=C^{n}(G, V)\oplus C^{n}(\huaB, T).
\end{equation*}
The coboundary operator $\delta: C^{n}(G, \huaB, V, T)\to C^{n+1}(G, \huaB, V, T)$ is defined by
\begin{equation}\label{defidgco}
\delta (\alpha_n, \beta_n)=(\dM^{\Theta}\alpha_n, \dM^{\huaB}\beta_{n-1}+c_n(\alpha_n)),
\end{equation}
where $\alpha_n\in C^{n}(G, V)$ and $\beta_{n-1}\in C^{n}(\huaB, T)$.

\begin{thm}
	With the above notations, $(\oplus_{i=1}^{+\infty}C^{i}(G, \huaB, V, T), \delta)$ is a cochain complex, i.e. $\delta\circ\delta=0$.
\end{thm}
\begin{proof}
This follows from Proposition \ref{chto}.
\end{proof}

\begin{defi}
The cohomology of the cochain complex $(\oplus_{i=1}^{+\infty}C^{i}(G, \huaB, V, T), \delta)$ is called
the cohomology of the Rota-Baxter Lie group $(G, \huaB)$ with coefficients in the representation 
$(V, T, \Theta)$. The corresponding $n$-th cohomology group is denoted by $H^n(G, \huaB, V, T)$.
\end{defi}

We note that in \cite{DR}, the authors constructed the chain complex for Rota–Baxter Lie groups only up to degree $3$, together with the first and second cohomology groups; as an application, they used the second cohomology group to study abelian extensions of Rota–Baxter Lie groups. The chain complex introduced there differs from the one considered here: in \cite{DR}, $V$ is equipped with a right $G$-module structure, whereas in our setting $V$ carries a left $G$-module structure. Moreover, we construct the full cochain complex $(\oplus_{i=1}^{+\infty}C^{i}(G, \huaB, V, T), \delta)$ in arbitrary degree, rather than restricting to low-degree terms. In subsequent sections, we will also use the cohomology obtained here to study central extensions of Rota-Baxter Lie groups.

The relation between $(\oplus_{i=1}^{+\infty}C^{i}(G, V), \dM^\Theta), (\oplus_{i=1}^{+\infty}C^{i}(\huaB, T), \dM^\huaB)$ and $(\oplus_{i=1}^{+\infty}C^{i}(G, \huaB, V, T), \delta)$ is given by the following theorem, which 
is resemblance of the Mayer-Vietoris sequence.

\begin{thm}\label{logco}
	There is a short exact sequence of the  cochain complexes:
	$$
	0\longrightarrow(\oplus_{n=1}^{+\infty}C^{n}(\huaB, T ),\dM^{\huaB})\stackrel{\frak{i}}{\longrightarrow}(\oplus_{n=1}^{+\infty}C^n(G, \huaB, V, T), \delta)\stackrel{\frak{p}}{\longrightarrow} (\oplus_{n=1}^{+\infty}C^n(G, V),\dM^{\Theta})\longrightarrow 0,
	$$
	where $\frak{i}(\beta_{n-1})=(0,\beta_{n-1})$ and $\frak{p}(\alpha_{n},\beta_{n-1})=\alpha_{n}$ for all $\beta_{n-1}\in C^{n}(\huaB, T)$ and $\alpha_{n}\in C^{n}(G, V)$.
	
	Consequently,
	there is a long exact sequence of the  cohomology groups:
	$$
	\cdots\longrightarrow H^{n}(\huaB, T)\stackrel{\frak{i}_{*}}{\longrightarrow} H^n(G,\huaB, V, T)\stackrel{\frak{p}_{*}}{\longrightarrow} H^n(G, V)\stackrel{\frkk^n}\longrightarrow H^{n+1}(\huaB, T)\longrightarrow\cdots,
	$$
	where the connecting map $\frkk^n$ is defined by
	\begin{equation}\label{connmap}
		\frkk^n([\alpha_n])=[c_n(\alpha_n)],\quad  \forall [\alpha_n]\in H^{n}(G, V).
	\end{equation}
\end{thm}
\begin{proof}
	By \eqref{defidgco}, we have the short exact sequence  of a chain complexes which induces a long exact sequence of cohomology groups. Moreover, if $\dM^{\Theta}\alpha_n=0$, then we can chose $(\alpha_n, 0)\in C^{n}(G, \huaB, V, T)$ such that $\frak{p}(\alpha_n, 0)=\alpha_n$. Since $\delta(\alpha_n, 0)=(\dM^{\Theta}\alpha_n, c_n(\alpha_n))$, it follows that $\frkk^{n}([\alpha_n])=[c_n(\alpha_n)]$.
\end{proof}

\section{A van Est theorem for cohomology of Rota-Baxter Lie groups}\label{secvan}
In this section, we establish the van Est theorem relating the cohomology of Rota-Baxter Lie groups to that of Rota-Baxter Lie algebras, which can be viewed as a justification for our cohomology theory of Rota-Baxter Lie groups. The classical van Est theorem \cite{Van} relates the smooth cohomology of a Lie group to the cohomology of its associated Lie algebra. See \cite{AC, CD, Cr, Hou, JSZ, JLS, Li, MS, Pfl, WX} for various generalizations of the van Est theorem.

Let $(\g, B)$ be a Rota-Baxter Lie algebra. In \cite{DAS}, the authors introduced a cohomology theory for Rota-Baxter Lie algebras, which we briefly review here.

Let $(V, T, \theta)$ be a representation of $(\g, B)$.  Define the space of $1$-cochains $C^1(\g, B, V, T)$ to be $\Hom(\g,V)$. For $n\geq 2$, define the space of $n$-cochains $C^n(\g, B, V, T)$  by
$$
C^n(\g,B, V, T)=C^{n}(\g, V)\oplus C^{n}(B, T),
$$
where $C^{n}(\g, V)=\Hom(\wedge^{n}\g,V)$, and $C^{1}(B, T)=0, C^{n}(B, T)=\Hom(\wedge^{n-1}\g ,V)$ for $n\geq2$.

Define the coboundary operator
\begin{equation*}
	\dM: C^n(\g,B, V, T)\lon C^{n+1}(\g, B, V, T)
\end{equation*}
by
\begin{equation}\label{deficodlg}
	\dM(\zeta_n, \xi_{n-1})=(\dM^{\theta}\zeta_n, \dM^{\theta_{B}}\xi_{n-1}+K_n(\zeta_n)),
\end{equation}
for all $\zeta_n\in\Hom(\wedge^{n}\g,V), \xi_{n-1}\in\Hom(\wedge^{n-1}\g,V)$,
where $\dM^{\theta}$ and $\dM^{\theta_B}$ are the Chevalley-Eilenberg coboundary operators of the Lie algebra $\g$ with coefficients in the representation $(V, \theta)$ and the Lie algebra $\g_B$ with coefficients in the representation $(V, \theta_B)$ respectively. Here $\theta_B: \g\lon\gl(V)$ is the representation of $\g_B$ on $V$ defined by $$\theta_B(x)u=\theta(B(x))u-T(\theta(x)u),\quad \forall x\in\g, u\in V,$$
and $K_n: \Hom(\wedge^{n}\g,V)\rightarrow\Hom(\wedge^{n}\g, V)$ is defined by
\begin{eqnarray*}
	&&K_n(\zeta_{n})(x_1,\cdots,x_{n})\\
	\nonumber&=&(-1)^{n}\Big(\zeta_n(B(x_1), \cdots, B(x_n))-\sum_{k=0}^{n-1}\sum_{1\leq i_1<\cdots<i_k\leq n}T\Big(\zeta_n(x_1,\cdots,x_{i-1},B(x_{i_1}),\cdots,B(x_{i_k}),\cdots,x_n)\Big)\Big).
\end{eqnarray*}
The corresponding $n$-th cohomology group is denoted by $H^{n}(\g, B, V, T)$. Moreover, we denote $n$-th cohomology groups of $(\oplus_{n=1}^{+\infty}C^{n}(\g, V), \dM^{\theta})$ and $(\oplus_{n=1}^{\infty}C^{n}(B, T), \dM^{\theta_B})$ by $H^{n}(\g, V)$ and $H^{n}(B, T)$ respectively. Then there is the following theorem.

\begin{thm}{\rm (\cite{DAS})}
	With the above notations, there is a short exact sequence of the  cochain complexes:
	$$
	0\longrightarrow(\oplus_{n=1}^{+\infty}C^{n}(B, T),\dM^{\theta_B})\stackrel{\iota}{\longrightarrow}(\oplus_{n=1}^{+\infty}C^n(\g, B, V, T), \dM)\stackrel{p}{\longrightarrow} (\oplus_{n=1}^{+\infty}C^{n}(\g, V),\dM^{\theta})\longrightarrow 0,
	$$
	where $\iota(\xi_{n-1})=(0,\xi_{n-1})$ and $p(\zeta_{n},\xi_{n-1})=\zeta_{n}$ for all $\xi_{n-1}\in \Hom(\wedge^{n-1}\g, V)$ and $\zeta_{n}\in \Hom(\wedge^{n}\g, V)$.
	
	Consequently,
	there is a long exact sequence of the  cohomology groups:
	$$
	\cdots\longrightarrow H^{n}(B, T)\stackrel{\iota_{*}}{\longrightarrow} H^n(\g, B, V, T)\stackrel{p_{*}}{\longrightarrow} H^n(\g, V)\stackrel{k^n}\longrightarrow H^{n+1}(B, T)\longrightarrow\cdots,
	$$
	where the connecting map $k^n$ is defined by
	\begin{equation}\label{tranco}
		k^n([\zeta_n])=[K_n(\zeta_n)],  \quad \forall [\zeta_n]\in H^{n}(\g, V).
	\end{equation}
\end{thm}
\emptycomment{
\begin{proof}
	By \eqref{deficodlg}, we have the short exact sequence  of chain complexes which induces a long exact sequence of cohomology groups. Moreover, if $\dM^{\theta}\zeta_n=0$, then we can chose $(\zeta_n, 0)\in C^{n}(\g, B, V, T)$ such that $p(\zeta_n, 0)=0$, which implies that $\dM(\zeta_n, 0)=(\dM^{\theta}\zeta_n, K(\zeta_n))$. Thus $k^{n}([\zeta_n])=[K(\zeta_n)]$.
\end{proof}
}
Let $G$ be a Lie group with Lie algebra $\g$, and let
$\Pi:G\to \GL(V)$ be a representation of $G$ with induced Lie algebra representation $\pi:\g\to\gl(V)$.
Define the van Est map $$\mathrm{VE}_n: C^{n}(G, V)\lon C^{n}(\g, V)$$ by
\begin{eqnarray*}
	&&\mathrm{VE}_n(\alpha_n)(x_1, \cdots, x_n)\\
	&=&\sum_{\epsilon\in S_{n}}(-1)^{|\epsilon|}\frac{d}{dt_{\epsilon(1)}}\cdots\frac{d}{dt_{\epsilon(n)}}\bigg|_{t_{\epsilon(1)}=\cdots=t_{\epsilon(n)}=0}\alpha_{n}\Big(\exp(t_{\epsilon(1)}x_{\epsilon(1)}), \cdots, \exp(t_{\epsilon(n)}x_{\epsilon(n)})\Big), \quad \forall x_1,\cdots, x_n\in\g.
\end{eqnarray*}
As in the classical theory relating the cohomology of Lie groups and Lie algebras, the collection $\mathrm{VE}_n$ assembles into a cochain map
 $$\mathrm{VE}: \oplus_{n=1}^{\infty} C^{n}(G, V)\lon\oplus_{n=1}^{\infty}C^{n}(\g, V),$$ which induces   homomorphisms ${\mathrm{VE}_{n}}_{*}:H^{n}(G, V)\to H^{n}(\g,, V)$ on the level of cohomology. Under suitable conditions on $G$, these homomorphisms are in fact isomorphisms, as the following classical theorem of van Est shows.
 
 \begin{thm}[van Est, \cite{Van}]\label{van} 
 	Let $G$ be a connected Lie group such that $\pi_i(G)=0$ for $1\leq i\leq n$. Then for every $1\leq i\leq n$, the van Est map induces an isomorphism
 	 \begin{equation*} 
 	 	H^{i}(G,V)\cong H^{i}(\g,V). 
 	 	\end{equation*} 
 \end{thm}
 
Let $(G, \huaB)$ be a Rota-Baxter Lie group and let $(\g, B)$ be the corresponding Rota-Baxter Lie algebra. Let $(V, T, \Theta)$ be a representation of $(G, \huaB)$. By Proposition \ref{diffr}, $(V, T, \theta)$ is a representation of $(\g, D)$. Moreover, by Proposition \ref{defirB}, the map $\Theta_\huaB$ defined by 
$$
\Theta_{\huaB}(a)u=(\Id_V+T)\Theta(\huaB(a))u-T\Big(\Theta(a\cdot_G\huaB(a))u\Big), \quad \forall a\in G, u\in V,
$$ 
is a representation of $G_\huaB$ on $V$. A direct computation shows that its differential at the identity satisfies 
$$(\Theta_\huaB)_{* e_G}(x)u=\theta(B(x))u-T(\theta(x)u), \quad \forall x\in\g, u\in V, $$
which is a representation of $\g_B$ on $V$.

With these representations in hand, we define $\widetilde{\mathrm{VE}}_{n}: C^{n}(G, \huaB, V, T)\lon C^{n}(\g, B, V, T)$ by
$$
\widetilde{\mathrm{VE}}_{n}(\alpha_n, \beta_{n-1})=(\mathrm{VE}_{n}(\alpha_n), \mathrm{VE}_{n-1}(\beta_{n-1})),\quad \forall (\alpha_n, \beta_{n-1})\in C^{n}(G, \huaB, V, T).
$$
Then we have the following theorem.

\begin{thm}\label{chainhomtopy}
The map $\widetilde{\mathrm{VE}}: \oplus_{n=1}^{+\infty}C^{n}(G, \huaB, V, T)\lon \oplus_{n=1}^{+\infty}C^{n}(\g, B, V, T)$ is a cochain map, which induces homomorphisms $\widetilde{\mathrm{VE}}_{n*}$ from the cohomology group $H^{n}(G, \huaB, V, T)$ to $H^{n}(\g, B, V, T)$. The map $\widetilde{\mathrm{VE}}$ is called the {\bf van Est map}.
\end{thm}
\begin{proof}
	For $n\geq1, \alpha_n\in C^{n}(G, V),  \beta_{n-1}\in C^{n}(\huaB, T)$, following the classical argument used for the cohomology of Lie groups and Lie algebras, we have
	\begin{eqnarray*}
		\widetilde{\mathrm{VE}}_{n+1}(\dM^{\Theta}\alpha_n, c_n(\alpha_n)+\dM^\huaB\beta_{n-1})&=&(\mathrm{VE}_{n+1}(\dM^{\Theta}\alpha_n), \mathrm{VE}_{n}(c_n(\alpha_n))+\mathrm{VE}_{n}(\dM^\huaB\beta_{n-1}))\\
		&=&\Big(\dM^{\theta}(\mathrm{VE}_{n}(\alpha_n)), \mathrm{VE}_{n}(c_n(\alpha_n))+\dM^{\theta_{B}}(\mathrm{VE}_{n-1}(\beta_{n-1}))\Big).
	\end{eqnarray*}
On the other hand, by definition of $\dM$, we also have
\begin{equation*}
\dM(\widetilde{\mathrm{VE}}_{n}(\alpha_n, \beta_{n-1}))=(\dM^{\theta}(\mathrm{VE}_{n}(\alpha_n)), K_n(\mathrm{VE}_{n}(\alpha_n))+\dM^{\theta_B}(\mathrm{VE}_{n-1}(\beta_{n-1}))),
\end{equation*}
thus to prove that $\widetilde{\mathrm{VE}}$ is a cochain map, it suffices to prove  $\mathrm{VE}_n(c_n(\alpha_n))=K_n(\mathrm{VE}_n(\alpha_n))$.

For any $x_i\in\g$, $1\leq i\leq n$, we have
	\begin{align*}
		&\quad \mathrm{VE}_n(c_n(\alpha_n))(x_1, \cdots, x_n)\\
		&=\sum_{\epsilon\in S_{n}}(-1)^{|\epsilon|}\frac{d}{dt_{\epsilon(1)}}\cdots\frac{d}{dt_{\epsilon(n)}}\bigg|_{t_{\epsilon(1)}=\cdots=t_{\epsilon(n)}=0}(-1)^n\Big((\Id_V+T)\alpha(\huaB(\exp(t_{\epsilon(1)}x_{\epsilon(1)})), \cdots, \huaB(\exp(t_{\epsilon(n)}x_{\epsilon(n)})))\\
		&\quad-T(\alpha_n(\exp(t_{\epsilon(1)}x_{\epsilon(1)})\cdot_G\huaB(\exp(t_{\epsilon(1)}x_{\epsilon(1)})),\cdots, \exp(t_{\epsilon(n)}x_{\epsilon(n)})\cdot_G\huaB(\exp(t_{\epsilon(n)}x_{\epsilon(n)}))))\Big)\\
		&=(-1)^{n}\Big((\Id_V+T)\sum_{\epsilon\in S_{n}}(-1)^{|\epsilon|}\frac{d}{dt_{\epsilon(1)}}\cdots\frac{d}{dt_{\epsilon(n)}}\bigg|_{t_{\epsilon(1)}=\cdots=t_{\epsilon(n)}=0}\alpha_n\Big(\exp(t_{\epsilon(1)}B(x_{\epsilon(1)})), \cdots, \exp(t_{\epsilon(n)}B(x_{\epsilon(n)}))\Big)\\
		&\quad-T(\sum_{k=0}^{n}\sum_{1\leq i_1<\cdots<i_k\leq n}\sum_{\epsilon\in S_{n}}(-1)^{|\epsilon|}\frac{d}{dt_{\epsilon(1)}}\cdots\frac{d}{dt_{\epsilon(n)}}\bigg|_{t_{\epsilon(1)}=\cdots=t_{\epsilon(n)}=0}\\
		&\quad\alpha_n(\exp(t_{\epsilon(1)}x_{\epsilon(1)}),\cdots\exp(t_{\epsilon(i_1)}B(x_{\epsilon(i_1)})),\cdots,\exp(t_{\epsilon(i_n)}B(x_{\epsilon(i_n)})),\cdots, \exp(t_{\epsilon(n)}x_{\epsilon(n)})))\Big)\\
		&=(-1)^{n}\Big(\mathrm{VE}_n(\alpha_n)(B(x_1), \cdots, B(x_n))-T(\sum_{k=0}^{n-1}\sum_{1\leq i_1<\cdots<i_k\leq n}\mathrm{VE}_n(\alpha_n)(x_1, \cdots, B(x_{i_1}), \cdots, B(x_{i_n}), \cdots, x_n))\Big)\\
		&=K_n(\mathrm{VE}_n(\alpha_n))(x_1, \cdots, x_n),
	\end{align*}
	 which implies that $\widetilde{\mathrm{VE}}$ is a cochain map. Therefore $\widetilde{\mathrm{VE}}_{n*}$ are homomorphisms from the cohomology group $H^{n}(G, \huaB, V, T)$ to $H^{n}(\g, B, V, T)$.
\end{proof}

\begin{thm}\label{van1}
	Assume that $G$ is a connected Lie group and its homotopy groups are trivial in $1, \cdots, n$. Then for $1\leq i\leq n$, the cohomology group $H^{i}(G, \huaB, V, T)$ is isomorphic to the cohomology group $H^{i}(\g, B, V, T)$.
\end{thm}
\begin{proof}
		For $[\alpha_{i-1}]\in H^{i-1}(G, V), [\beta_{i-1}]\in H^{i}(\huaB, T)$ and $[(\alpha_i, \beta'_{i-1})]\in H^{i}(G, \huaB, V, T)$, by Theorem \ref{chainhomtopy}, we have
		\begin{eqnarray*}
			k^{i-1}(\mathrm{VE}_{i-1*}([\alpha_{i-1}]))&=&[K_{i-1}(\mathrm{VE}_{i-1}(\alpha_{i-1}))]=\mathrm{VE}_{i-1*}(\frkk^{i-1}([\alpha_{i-1}])),\\
			\widetilde{\mathrm{VE}}_{i*}(\frak{i}_*([\beta_{i-1}]))&=&[(0, \mathrm{VE}_{i-1}(\beta_{i-1}))]=\iota_*(\mathrm{VE}_{i-1*}([\beta_{i-1}])),\\
			\mathrm{VE}_{i*}(\frak{p}_*([(\alpha_i, \beta'_{i-1})]))&=&[\mathrm{VE}_i(\alpha_{i})]=p_*(\widetilde{\mathrm{VE}}_{i*}([(\alpha_{i}, \beta'_{i-1})])),
		\end{eqnarray*}
		where $k^{i-1}$ and $\frkk^{i-1}$ are given by \eqref{tranco} and \eqref{connmap} respectively.  Thus we have the following commutative diagram:
		\[\begin{CD}
			H^{i-1}(G, V)@>\frkk^{i-1}>>H^{i}(\huaB, T)@>\frak{i}_*>>H^{i}(G, \huaB, V, T)@>\frak{p}_*>>H^{i}(G, V)            @>\frkk^{i}>>H^{i+1}(\huaB, T)\\
			@V \mathrm{VE}_{i-1*} VV    @V \mathrm{VE}_{i-1*} VV   @V\widetilde{\mathrm{VE}}_{i*}VV  @V \mathrm{VE}_{i*} VV    @V \mathrm{VE}_{i+1*} VV\\
			H^{i-1}(\g, V)@>k^{i-1}>>H^{i}(B, T) @>\iota_*>>H^{i}(\g, B, V, T)@>p_*>>H^{i}(\g, V)             @>k^{i}>>H^{i+1}(B, T)
			.
		\end{CD}\]
		Since $G$ is connected and its homotopy groups are trivial in $1, \cdots, n$, for $1\leq i\leq n$, by Theorem \ref{van}, we have the following group isomorphism $H^{i}(\huaB, T)\simeq H^{i}(B, T), H^{i+1}(\huaB, T)\simeq H^{i+1}(B, T)$ and $H^{i-1}(G, V)\simeq H^{i-1}(\g, V), H^{i}(G, V)\simeq H^{i}(\g, V)$. Apply the Five Lemma to the above diagram, we have
		$$H^{i}(G, \huaB, V, T)\simeq H^{i}(\g, B, V, T), \quad \forall 1\leq i\leq n.$$
		Thus, for $1\leq i\leq n$, the cohomology group $H^{i}(G, \huaB, V, T)$ is isomorphic to the cohomology group $H^{i}(\g, B, V, T)$.
\end{proof}

\section{Central extensions of Rota-Baxter Lie groups}\label{secch2}
In this section, we apply the cohomology theory developed in Section \ref{seccoh} to classify central extensions of Rota-Baxter Lie groups, showing that they are classified by the second cohomology group.

\begin{defi}{\rm(\cite{BRS})}
	Let $(G, \huaB)$ be a Rota-Baxter Lie group and $V$ a vector space, $T:V\to V$ a linear map. A {\bf central extension} of $(G, \huaB)$ by $(V, T)$ is given by a short exact sequence of Rota-Baxter Lie group homomorphisms:
	\begin{equation*}
		{0} \longrightarrow (V, T) \stackrel{i}{\hookrightarrow} (H, \huaB_H) \stackrel{p}{\twoheadrightarrow} (G, \huaB) \longrightarrow {0},
	\end{equation*}
such that $i(V)$ is contained in the center $Z(H)$ of the Lie group $(H, \cdot_H, e_H)$, where $(H, \huaB_H)$ is a Rota-Baxter Lie group.
\end{defi}

\begin{defi}{\rm(\cite{BRS})}
	Let $(H, \huaB_{H})$ and $(H', \huaB_{H'})$ be two central extensions of a Rota-Baxter Lie group $(G, \huaB)$ by $(V, T)$. They are said to be {\bf isomorphic} if there exists an isomorphism $\sigma: H\to H'$ of Rota-Baxter Lie groups such that the following diagram commutes:
	\begin{equation*}
		\xymatrix{
			{0} \ar[r] & (V, T) \ar[d]_{\Id_V} \ar@{^{(}->}[r]^{i} & (H, \huaB_H) \ar[d]_{\sigma} \ar@{->>}[r]^{p} & (G, \huaB) \ar[d]_{\Id_G} \ar[r] & {0} \\
			{0} \ar[r] & (V, T) \ar@{^{(}->}[r]^{i'} & (H', \huaB_{H'}) \ar@{->>}[r]^{p'} & (G, \huaB) \ar[r] & {0}.
		}
	\end{equation*}
	Denote by $\mathrm{CExt(G, \huaB, V, T)}$ the set of equivalence classes of central extensions of the Rota-Baxter Lie group $(G, \huaB)$ by $(V, T)$.
	\end{defi}
	
	\begin{defi}
		A {\bf section} of a central extension $(H, \huaB_H)$ of a Rota-Baxter Lie group $(G, \huaB)$ by $(V, T)$ is a smooth map $s: G\lon H$ such that
		\begin{equation*}
			p\circ s=\Id_G, \quad s(e_G)=e_H.
		\end{equation*}
	\end{defi}

Let $(H, \huaB_H)$ be a central extension of $(G, \huaB)$ by $(V, T)$. We identify $V$ with $i(V)$. Assume that $s:G\to H$ is a section. Define a smooth map $\tau: G\times V\to H$ by
\begin{equation*}
	\tau(a, u)=u\cdot_H s(a), \quad \forall a\in G, u\in V.
\end{equation*}
The inverse of $\tau$ is given by
\begin{equation*}
	\tau^{-1}(h)=\Big(p(h), h\cdot_H(s(p(h)))^{-1}\Big), \quad \forall h\in H.
\end{equation*}
Using $\tau$, we can transport the Rota--Baxter Lie group structure on $H$ to $G \times V$: the multiplication is
\begin{equation*}
	(a, u)\cdot(b, v)=\tau^{-1}\Big(\tau(a, u)\cdot_H\tau(b, v)\Big)=(a\cdot_G b, u+v+\alpha(a, b)),\quad\forall a, b\in G, u,v\in V,
\end{equation*}
and the Rota--Baxter operator is
\begin{equation}\label{defibbb}
	\mathbb{B}(a, u)=\tau^{-1}\Big(\huaB_H(\tau(a, u))\Big)=\Big(\huaB(a), \huaB_H(u\cdot_H s(a))\cdot_H(s(\huaB(a)))^{-1}\Big), \quad\forall a\in G, u\in V,
\end{equation}
where 
\begin{equation*}
\alpha(a, b)=s(a)\cdot_H s(b)\cdot_H (s(a\cdot_G b))^{-1}.
\end{equation*}
Consequently, $(G\times V, \mathbb{B})\cong (H, \huaB_H)$. We now give the explicit formula for the Rota–Baxter operator $\mathbb{B}$ as follows.

\begin{pro}
	If $s:G\to H$ is a section, then for any $a\in G, u\in V$,
	\begin{equation*}
		\mathbb{B}(a, u)=\Big(\huaB(a), T(u)+\beta(a)\Big),
	\end{equation*} 
	where  $\beta(a)=\huaB_H(s(a))\cdot_H \Big(s(\huaB(a))\Big)^{-1}$.
\end{pro}
\begin{proof}
By \eqref{defibbb}, we have
\begin{equation*}
\mathbb{B}(a, u)=\tau^{-1}\Big(\huaB_H(\tau(a, u))\Big)=\Big(\huaB(a), \huaB_H(u\cdot_H s(a))\cdot_H(s(\huaB(a)))^{-1}\Big), \quad\forall a\in G, u\in V.
\end{equation*}
Let $R_a(u)=\huaB_H(u\cdot_H s(a))\cdot_H(s(\huaB(a)))^{-1}$. On the one hand, for any $u, v\in V$ and $a\in G$, 
\begin{align*}
	\mathbb{B}(e_G, v)\cdot\mathbb{B}(a, u)&=(\huaB(a), T(v)+R_a(u)).
\end{align*}
On the other hand, 
\begin{align*}
\mathbb{B}\Big((e_G, v)\cdot(e_G, T(v))\cdot(a, u)\cdot(e_G, -T(v))\Big)=\Big(\huaB(a), R_a(v+u)\Big).
\end{align*}
Comparing the second components of the two expressions above yields
\begin{equation*}
	R_a(v)=T(v)+\huaB_H(s(a))\cdot_H \Big(s(\huaB(a))\Big)^{-1}.
\end{equation*}
Thus,
$
	\mathbb{B}(a, u)=\Big(\huaB(a), T(u)+\beta(a)\Big),
$
where  $\beta(a)=\huaB_H(s(a))\cdot_H \Big(s(\huaB(a))\Big)^{-1}$.
\end{proof}

\begin{thm}\label{thmcon123}
	If $s:G\to H$ is a section, then the pair $(\alpha, \gamma)$ defined by
	\begin{equation*}
		 \alpha(a, b)=s(a)\cdot_H s(b)\cdot_H (s(a\cdot_G b))^{-1},\quad	\gamma(a)=\huaB_H(s(a))\cdot_H \Big(s(\huaB(a))\Big)^{-1}-T\Big(\alpha(a, \huaB(a))\Big),
	\end{equation*}
for any $a, b\in G$, is a $2$-cocycle of the Rota-Baxter Lie group $(G, \huaB)$ with coefficients in the representation $(V, T, \Id_V)$.

Moreover, its cohomological class 
	does not depend on the choice of sections.
\end{thm}
\begin{proof}
By the fact that $(G\times V, \cdot)$ is a Lie group , where the group multiplication $\cdot$ is given by
\begin{equation*}
(a, u)\cdot (b, v)=(a\cdot_G b, u+v+\alpha(a, b)), \quad \forall a, b\in G, u, v\in V.
\end{equation*}
We deduce that $\alpha$ is a $2$-cocycle of $(G, \cdot_G)$ with coefficients in $(V, \Id_V)$.

Moreover, since $\mathbb{B}(a, u)=\Big(\huaB(a), T(u)+\beta(a)\Big)$ is a Rota-Baxter operator, we obtain on the one hand,
\begin{align*}
\mathbb{B}(a, u)\cdot\mathbb{B}(b, v)=\Big(\huaB(a)\cdot_G\huaB(b), T(u)+T(v)+\beta(a)+\beta(b)+\alpha(\huaB(a), \huaB(b))\Big).
\end{align*}
On the other hand,
\begin{align*}
&\quad\mathbb{B}\Big((a, u)\cdot \mathbb{B}(a, u)\cdot (b, v)\cdot\mathbb{B}(a, u)^{-1}\Big)\\
&=\Big(\huaB(a)\cdot_G\huaB(b), T(u)+T(v)+T\Big(\alpha(a, \huaB(a))\Big)+T\Big(\alpha(a\cdot_G\huaB(a), b)\Big)\\
&\quad-T\Big(\alpha(\huaB(a), \huaB(a)^{-1})\Big)+T\Big(\alpha(a\cdot_G\huaB(a)\cdot_G b, \huaB(a)^{-1})\Big)+\beta(a\cdot_G\huaB(a)\cdot_G b\cdot_G\huaB(a)^{-1})\Big).
\end{align*}
Since $\dM^{\Id_V}\alpha(a\cdot_G\huaB(a)\cdot_G b\cdot_G\huaB(a)^{-1}, \huaB(a), \huaB(a)^{-1})=0$, we have 
\begin{equation*}
\alpha(\huaB(a), \huaB(a)^{-1})-\alpha(a\cdot_G\huaB(a)\cdot_G b, \huaB(a)^{-1})-\alpha(a\cdot_G\huaB(a)\cdot_G b\cdot_G\huaB(a)^{-1}, \huaB(a))=0.
\end{equation*}
Therefore,
\begin{align}
\label{defieqga}&\quad\beta(a)+\beta(b)+\alpha(\huaB(a), \huaB(b))-\beta(a\cdot_G\huaB(a)\cdot_G b\cdot_G\huaB(a)^{-1})\\
\nonumber&=T\Big(\alpha(a, \huaB(a))+\alpha(a\cdot_G\huaB(a), b)-\alpha(a\cdot_G\huaB(a)\cdot_G b\cdot_G\huaB(a)^{-1}, \huaB(a))\Big).	
\end{align}
By \eqref{repB}, the associated action $\Theta_B: G\to\GL(V)$ is given by
$
\Theta_B(a)u=u.	
$
 Since $\gamma(a)=\beta(a)-T\Big(\alpha(a, \huaB(a))\Big),$ by \eqref{defieqga} we have 
\begin{align*}
&\quad\dM^{\huaB}(\gamma)+c_2(\alpha)\\
&=\beta(b)-\beta(a\cdot_G\huaB(a)\cdot_G b\cdot_G\huaB(a)^{-1})+\beta(a)-T\Big(\alpha(a, \huaB(a))\Big)-T\Big(\alpha(b, \huaB(b))\Big)\\
&\quad+T\Big(\alpha(a\cdot_G\huaB(a)\cdot_G b\cdot_G\huaB(a)^{-1}, \huaB(a)\cdot_G\huaB(b))\Big)+\alpha(\huaB(a), \huaB(b))+T(\alpha(\huaB(a), \huaB(b)))\\
&\quad-T(\alpha(a\cdot_G\huaB(a), b\cdot_G\huaB(b)))\\
&=T\Big(\alpha(a\cdot_G\huaB(a), b)-\alpha(a\cdot_G\huaB(a)\cdot_G b\cdot_G\huaB(a)^{-1}, \huaB(a))-\alpha(b, \huaB(b))\\
&\quad+\alpha(a\cdot_G\huaB(a)\cdot_G b\cdot_G\huaB(a)^{-1}, \huaB(a)\cdot_G\huaB(b))+\alpha(\huaB(a), \huaB(b))-\alpha(a\cdot_G\huaB(a), b\cdot_G\huaB(b))\Big).
\end{align*}
Using the identities $\dM^{\Id_V}\alpha(a\cdot_G\huaB(a)\cdot_G b\cdot_G\huaB(a)^{-1}, \huaB(a), \huaB(b))=0$ and $\dM^{\Id_V}\alpha(a\cdot_G\huaB(a), b, \huaB(b))=0$, we obtain  
\begin{equation*}
	\alpha(a\cdot_G\huaB(a)\cdot_G b\cdot_G\huaB(a)^{-1}, \huaB(a))+\alpha(a\cdot_G\huaB(a)\cdot_G b, \huaB(b))=\alpha(\huaB(a), \huaB(b))+\alpha(a\cdot_G\huaB(a)\cdot_G b\cdot_G\huaB(a)^{-1}, \huaB(a)\cdot_G\huaB(b)),
\end{equation*}
and
\begin{equation*}
	\alpha(a\cdot_G\huaB(a), b)+\alpha(a\cdot_G\huaB(a)\cdot_G b, \huaB(b))=\alpha(b, \huaB(b))+\alpha(a\cdot_G\huaB(a), b\cdot_G\huaB(b)),
\end{equation*}
which imply that $(\alpha, \gamma)$ is a $2$-cocycle of the Rota-Baxter Lie group $(G, \huaB)$ with coefficients in the representation $(V, T, \Id_V)$.

If $s':G\to H$ is another section, then there exists a smooth map $\lambda: G\to V$ such that $s'(a)=\lambda(a)\cdot_H s(a)$ for any $a\in G$. Define
\begin{equation*}
\alpha'(a, b)=s'(a)\cdot_G s'(b)\cdot_G (s'(a\cdot_G b))^{-1},\quad \gamma'(a)=\huaB_H(s'(a))\cdot_H\Big(s'(\huaB(a))\Big)^{-1}-T\Big(\alpha'(a, \huaB(a))\Big),
\end{equation*}
for any $a, b\in G$. By direct calculation, we obtain
\begin{equation*}
	\alpha'(a, b)-\alpha(a, b)=\lambda(b)-\lambda(a\cdot_G b)+\lambda(a)=\dM^{\Id_V}\lambda(a, b), \quad \forall a, b\in G.
\end{equation*}
Since $\lambda(a)\in Z(H)$, then $\huaB_H(\lambda(a)\cdot_H s(a))=T(\lambda(a))\cdot_H\huaB_H(s(a))$, which implies
\begin{equation*}
	\gamma'(a)-\gamma(a)=T\Big(\lambda(a\cdot_G\huaB(a))-\lambda(\huaB(a))\Big)-\lambda(\huaB(a)).
\end{equation*}
Therefore, $(\alpha', \gamma')-(\alpha, \gamma)=\dM\lambda$, which shows that cohomological class is independent of the choice of sections.
\end{proof}

Finally, we use the second cohomology groups to classify central extensions of Rota-Baxter Lie groups. Let $V$ be a vector space and $T:V\to V$ a linear map. Consider the representation $(V, T, \Id_V)$ of the Rota-Baxter Lie group $(G, \huaB)$.
\begin{thm}\label{cenH2}
The central extensions of $(G, \huaB)$ by $(V, T)$ are classified by $H^2(G, \huaB, V, T)$, where $H^2(G, \huaB, V, T)$ is the second cohomology group associated with the representation $(V, T, \Id_V)$. More precisely, there is a bijection
\begin{equation*}
\mathrm{CExt}(G,\huaB,V,T)\;\xrightarrow\;H^2(G,\huaB,V,T).
\end{equation*}
\end{thm}
\begin{proof}
Let $(H, \huaB_H)$ be a central extension of $(G, \huaB)$ by $(V, T)$. Let $s:G\to H$ be a section. By Theorem \ref{thmcon123}, the pair $(\alpha,\gamma)$ defined by
\begin{align*}
	\alpha(a,b)&=s(a)\cdot_H s(b)\cdot_H s(a\cdot_G b)^{-1},\\
	\gamma(a)&=\huaB_H(s(a))\cdot_H s(\huaB(a))^{-1}-T\big(\alpha(a,\huaB(a))\big)
\end{align*}
is a $2$-cocycle, and its cohomology class is independent of the choice of $s$. 

	Let $\sigma:H\to H'$ be an equivalence of extensions. Then $s'=\sigma\circ s$ is a section.
Since $\sigma\circ\huaB_H=\huaB_{H'}\circ\sigma$, for all $a,b\in G$,
\begin{align*}
	\alpha'(a,b)
	&=\sigma\big(s(a)\cdot_H s(b)\cdot_H s(a\cdot_G b)^{-1}\big)=\alpha(a,b),\\
	\gamma'(a)
	&=\sigma\big(\huaB_H(s(a))\cdot_H s(\huaB(a))^{-1}\big)
	-T\big(\alpha'(a,\huaB(a))\big)=\gamma(a).
\end{align*}
Hence the class $[(\alpha,\gamma)]$ depends only on the equivalence class of the extension, and
\begin{equation*}
\zeta:\mathrm{CExt}(G,\huaB,V,T)\longrightarrow H^2(G,\huaB,V,T),\qquad
\zeta\big([(H,\huaB_H)]\big)=[(\alpha,\gamma)],
\end{equation*}
is well defined.

Let $(\alpha,\gamma)$ be a $2$-cocycle. Define a multiplication on $G\times V$ by
\begin{equation*}
(a,u)\cdot(b,v)=\big(a\cdot_G b,\;u+v+\alpha(a,b)\big),\qquad a,b\in G,\ u,v\in V,
\end{equation*}
and a map $\mathbb B:G\times V\to G\times V$ by
\begin{equation*}
\mathbb B(a,u)=\Big(\huaB(a),\;T(u)+\gamma(a)+T\big(\alpha(a,\huaB(a))\big)\Big).
\end{equation*}
The cocycle conditions imply that $(G\times V,\cdot)$ is a group and that $\mathbb B$ satisfies the
Rota-Baxter identity. Moreover, $(G\times V,\mathbb B)$ is a central
extension of $(G,\huaB)$ by $(V,T)$.

Suppose $[(\alpha',\gamma')]=[(\alpha,\gamma)]$. Then there exists $\lambda:G\to V$ such that
	\begin{equation*}
	\lambda(e_G)=0, \quad 	\dM\lambda=(\alpha,\gamma)-(\alpha',\gamma')
	\end{equation*} 
Let $\cdot'$ and $\mathbb B'$ be the group structure and Rota-Baxter operator on $G\times V$ determined by $(\alpha',\gamma')$,
	and define
	\begin{equation*}
	\sigma:G\times V\to G\times V,\qquad \sigma(a,u)=\big(a,\,u+\lambda(a)\big).
	\end{equation*}
Since 	$\alpha(a,b)-\alpha'(a,b)=\lambda(a)+\lambda(b)-\lambda(a\cdot_G b)$, we obtain 
\begin{equation*}
	\sigma(a,u)\cdot'\sigma(b,v)=\big(a\cdot_G b,\;u+v+\lambda(a)+\lambda(b)+\alpha'(a,b)\big)=\big(a\cdot_G b,\;u+v+\alpha(a,b)+\lambda(a\cdot_G b)\big)=\sigma\big((a,u)\cdot(b,v)\big),
\end{equation*}
	Similarly, since $\gamma(a)-\gamma'(a)=T\lambda\big(a\cdot_G\huaB(a)\big)-\lambda(\huaB(a))-T\lambda(\huaB(a))$, we obtaion
$\sigma\circ\mathbb B=\mathbb B'\circ\sigma$.
	Thus $\sigma$ is an equivalence, and
	\begin{equation*}
	\eta:H^2(G,\huaB,V,T)\longrightarrow\mathrm{CExt}(G,\huaB,V,T),\qquad
	\eta\big([(\alpha,\gamma)]\big)=[(G\times V,\mathbb B)],
	\end{equation*}
	is well defined.
	
Let $[(\alpha,\gamma)]\in H^2(G,\huaB,V,T)$, and let
$(G\times V,\mathbb B)$ be the associated central extension of $(G,\huaB)$ by $(V,T)$. Take the section $s(a)=(a,0)$ of $(G\times V,\mathbb B)$.
	Then
	\begin{equation*}
	s(a)\cdot s(b)\cdot s(a\cdot_G b)^{-1}=\bigl(e_G,\alpha(a,b)\bigr),\qquad
	\mathbb B(s(a))\cdot s(\huaB(a))^{-1}=\bigl(e_G,\gamma(a)+T(\alpha(a,\huaB(a)))\bigr),
	\end{equation*}
	which means that $\zeta\circ \eta=\mathrm{id}$.
	
Let $(H,\huaB_H)$ be a central extension with section $s$ and
	associated cocycle $(\alpha,\gamma)$. The Rota-Baxter Lie group $(H,\huaB_H)$ is isomorphic to $(G\times V,\mathbb B)$, where the multiplication on $G\times V$ is
	\begin{equation*}
		(a, u)\cdot (b, v)=(a\cdot_G b, u+v+\alpha(a, b)), \quad \forall a,b\in G, u, v\in V
	\end{equation*}
	and $\mathbb B$ is given by
	\begin{equation*}
		\mathbb B(a, u)=\Big(\huaB(a), T(u)+\gamma(a)+T(\alpha(a, \huaB(a)))\Big).
	\end{equation*}
	Therefore, 
	$\eta\circ\zeta([(H,\huaB_H)])=[(H,\huaB_H)]$.
	
	Hence $\zeta$ is a bijection with inverse $\eta$, and the central extensions of $(G,\huaB)$ by
	$(V,T)$ are classified by $H^2(G,\huaB,V,T)$.
\end{proof}
\begin{rmk}
Belwal, Rathee and Singh \cite{BRS} developed a low-dimensional cohomology theory for relative Rota–Baxter groups and used its second cohomology group to classify abelian and central extensions. When the action is given by the adjoint action, their framework is closely related to that of Rota–Baxter groups. In contrast, we construct a cochain complex in arbitrary degrees via a mapping cone construction, which also can classify the central extensions of Rota-Baxter Lie groups.	It remains unclear whether this chain complex can be used to classify abelian Rota–Baxter group extensions.
\end{rmk}

\end{document}